\documentclass[11pt,rqno, a4paper]{article}

\usepackage{xcolor}
\pagecolor{white}

\usepackage{titlesec}
\titleformat*{\section}{\sc\centering\large} % INDICATES THE FONT SIZE OF THE SECTIONS. 
\titleformat*{\subsection}{\bf} % INDICATES THE FONT SIZE OF THE SUBSECTIONS. 
\titleformat*{\subsubsection}{\it} % INDICATES THE FONT OF SUBSECTIONS.

\usepackage{graphicx}%
\usepackage{multirow}%
\usepackage{amsmath,amssymb,amsfonts}%
\usepackage{amsthm}%
\usepackage{mathrsfs}%
\usepackage[title]{appendix}%
\usepackage{xcolor}%
\usepackage{textcomp}%
\usepackage{manyfoot}%
\usepackage{booktabs}%
\usepackage{algorithm}%
\usepackage{algorithmicx}%
\usepackage{algpseudocode}%
\usepackage{listings}%

\usepackage{mathtools}
\usepackage{multirow}
\usepackage{verbatim}
\usepackage{graphicx}
\usepackage[shortlabels]{enumitem}
\usepackage{hyperref}
\usepackage[noabbrev]{cleveref}
\usepackage{caption}
\usepackage[normalem]{ulem}
\usepackage{amsthm,amscd,upref,amstext}

\DeclareMathOperator{\Ric}{Ric}

\DeclarePairedDelimiterX{\norm}[1]{\lVert}{\rVert}{#1}
\usepackage{amsmath}
\usepackage{amssymb}
\usepackage{amsfonts}
\usepackage{setspace}
\usepackage{latexsym}
\usepackage{mathrsfs}
\usepackage{epsfig}
\usepackage{amsthm}
\usepackage{color}

\usepackage[margin=2cm]{caption}

\usepackage{comment}

\usepackage{transparent}
\graphicspath{{./}{images/}}

\usepackage{import}
\usepackage{xifthen}
\usepackage{pdfpages}
\usepackage{transparent}

\newcommand{\be}{\begin{equation}}
\newcommand{\ee}{\end{equation}}

\newcommand{\vs}{\vspace{0.2cm}}

\usepackage{fancyhdr}
\newtheorem{Theorem}{Theorem}[section]
\newtheorem{Remark}[Theorem]{Remark}
\newtheorem{Definition}[Theorem]{Definition}
\newtheorem{Proposition}[Theorem]{Proposition}

\newtheorem{Lemma}[Theorem]{Lemma}
\newtheorem{Corollary}[Theorem]{Corollary}

\usepackage{tocloft} % TABLE OF CONTENTS CUSTOMIZATION

\usepackage{setspace, tocloft}

\allowdisplaybreaks

\numberwithin{equation}{section}

\begin{document}
%\thispagestyle{empty}

%\begin{center}
%\setcounter{Maxaffil}{1}
%\renewcommand\Affilfont{\itshape}

\begin{center}
{\LARGE\bf Deriving Perelman's entropy from Colding's\vs 

monotonic volume}
\vspace{.5cm}

{\large\sc Ignacio Bustamante\footnote{ibustamante@fing.edu.uy} {\rm and}\ Mart\'in Reiris\footnote{mreiris@cmat.edu.uy}}
\vs

{\it Universidad de la Rep\'ublica, Uruguay}
\vspace{.2cm}

\begin{abstract}
In his groundbreaking work from 2002, Perelman introduced two fundamental monotonic quantities: the reduced volume and the entropy. While the reduced volume was motivated by the Bishop-Gromov volume comparison applied to a suitably constructed $N$-space, which becomes Ricci-flat as $N\rightarrow \infty$, Perelman did not provide a corresponding explanation for the origin of the entropy. In this article, we demonstrate that Perelman's entropy emerges as the limit of Colding's monotonic volume for harmonic functions on Ricci-flat manifolds, when appropriately applied to Perelman's $N$-space.    
\end{abstract}

\end{center}

\section{Introduction}

Monotonic quantities play a fundamental role in elliptic and parabolic PDEs, particularly in the study of singularities, regularity of solutions, and asymptotics \cite{CM12}. For instance, Huisken's monotonicity formula for the mean curvature flow \cite{H90} and Struwe's monotonicity formula for harmonic map heat flows \cite{S88} (see also Hamilton's generalization of both formulas to general manifolds \cite{H93}), are well-known examples of monotonic quantities for parabolic equations, which have been widely applied in geometric analysis. For the `elliptic counterparts' of these parabolic cases, the minimal surface equation and the harmonic map equation, there are analogous, though distinct, quantities, such as Allard's monotonicity for minimal surfaces \cite{A72} and the well-known monotonicity for harmonic maps \cite{Schoen1982}, respectively. Additional examples include Hamilton's monotonic formula for the Yang-Mills heat flow \cite{H93} and Price's monotonic formula for the elliptic Yang-Mills equation \cite{P83}, as well as Almgren's frequency for harmonic functions \cite{A83} and Poon's parabolic frequency for the heat equation  \cite{Po96} (see also the recent generalization of  the parabolic frequency to manifolds by Colding and Minicozzi \cite{CM21}). 

Parabolic monotonic formulas are often intriguing and challenging to derive and typically rely on a peculiar use of backward solutions to heat-type equations, whereas elliptic quantities do not. Despite appearing unrelated, Davey \cite{Dav18} (see also \cite{DS24}) demonstrated that, at least in a handful of examples that include some of the ones mentioned earlier, the parabolic monotonicity can be derived from a subtle use of elliptic ones. This is achieved by applying elliptic monotonicity to associated equations on suitably constructed $N$-dimensional spaces, and taking the limit as $N\rightarrow \infty$. These equations, which are equivalent to the parabolic equation (they hold if and only if the parabolic does) are formally identical to the elliptic one up to terms that vanish as $N\rightarrow \infty$. Put differently, the parabolic equation can be viewed as the elliptic one in an $N$-space plus negligible additional terms, allowing the monotonic parabolic formula to be derived as a limit of an elliptic one (see the related discussions by Tao \cite{T09l} and Šverák \cite{Sve11}). 
 
For the Ricci-flat equation,
\be\label{RIC=0}
\Ric = 0,
\ee
the best-known monotonic quantity is the Bishop-Gromov relative volume, widely applied across a variety of contexts in differential geometry. Another example, which will be central to this article, was introduced by Colding in \cite{C12}. This new quantity, which we will refer to as `monotonic volume', is defined along the level sets of positive Green functions and was used, for instance, to study asymptotic cones on Ricci-flat non-parabolic manifolds \cite{C12, CM13}. Generalizations of Colding's monotonic volume were later given by Colding and Minicozzi in \cite{CMi13} and applications to General Relativity were explored by Agostiniani, Mazzieri and Oronzio in \cite{AMO24}.

For the Ricci flow equation,
\be
\partial_{t}g = -2\Ric,
\ee
the `parabolic counterpart' of the Ricci flat equation, Perelman \cite{Per02} introduced two fundamental monotonic quantities: the reduced volume and the entropy. Additionally, he demonstrated that the reduced volume can be derived from the Bishop-Gromov relative volume on a carefully constructed $N$-space that becomes Ricci-flat as the dimension $N$ goes to infinity. However, no analogous justification on the origin of the entropy was provided.

In this article, we demonstrate that Perelman's entropy arises as the limit as $N\rightarrow \infty$ of Colding's monotonic volume when appropriately applied to Perelman's $N$-space. This proves that, as in the previously discussed parabolic examples, both the reduced volume and the entropy can be understood as originating from a unified framework using monotonic formulas from the `counterpart' Ricci-flat equation. Additionally, we derive the expression for the derivative of Perelman's entropy from Colding's formula for the derivative of the monotonic volume.

\subsection{Overview}

To explain the results we begin recalling the definition of Perelman's entropy, of Perelman's $N$-space, and of Colding's monotonic volume. 

Let us start recalling Perelman's entropy. On a closed manifold $M^{n}$ of dimension $n$, let $g(\tau)$, $\tau\in [0,T]$, be a solution to the backward Ricci flow equation,
\be\nonumber
\partial_{\tau} g = 2\Ric.
\ee
Then, let $u$ be a solution to,
\be
\partial_{\tau} u = \Delta u - Ru,
\ee
positive at the initial time $\tau = 0$, and hence, also positive for all times by the maximum principle. Define $f$ by $u = \tau^{-n/2}e^{-f}$ so that $f$ satisfies,
\be\label{BACKCHO}
\partial_{\tau} f = \Delta f - |\nabla f|^{2} + R - \frac{n}{2\tau}.
\ee
Then, the Perelman entropy $\mathcal{W}$ (for the function $f$) is given by,
\be\label{Def W}
\mathcal{W}(\tau) = \int_{M}\left(\tau(|\nabla f|^{2}+R) + f - n\right)(4\pi \tau)^{-n/2}e^{-f}d\nu,
\ee
and its derivative takes the form,
\be\label{derivative of W}
\frac{d}{d\tau} \mathcal{W} = - \int_{M}2\tau \left\vert \Ric + \nabla\nabla f -\frac{1}{2\tau}g\right\rvert^{2}(4\pi\tau)^{-n/2}e^{-f}d\nu,
\ee
from which it is deduced that it is monotonically decreasing in $\tau$.
\vs
\vs

Next, we define Perelman's $N$-space, which he used to motivate the definition of the reduced volume. As earlier, let $g(\tau)$ be a solution to the backward Ricci flow equation on $[0,T]$. Denote by $\mathbb{S}^{N}$ the $N$-dimensional unit-sphere of $\mathbb{R}^{N+1}$. Let $r$ be the distance to the origin in $\mathbb{R}^{N+1}$ and let $\theta$ denote points in $\mathbb{S}^{N}$. 
Then, the Perelman $N$-space $(\hat{M},\hat{g})$ is the manifold,  
\be\nonumber
\hat{M}^{m} := (0,\sqrt{2NT})_{r}\times \mathbb{S}^{N}_{\theta}\times M^{n}_{x}\subset \mathbb{R}^{N+1}\times M^{n},\quad {\rm where}\quad m=N+n+1,
\ee
endowed with the metric,
\be
\hat{g} := r^{2} g_{\mathbb{S}^{N}} + (1+\frac{Rr^{2}}{N^{2}})dr^{2} + g,
\ee
where $R$ is the scalar curvature of $g$ and where at a point $(r,\theta, x)\in \hat{M}$, $g_{\mathbb{S}^{N}}$ is evaluated at $\theta$ and $R$ and $g$ at $(\tau = r^{2}/2N,x)\in (0,T)\times M$. 
\vs
\vs

Now, we briefly discuss Colding's monotonic volume. On a manifold $(N,\bar{g})$ admitting a positive and proper Green function $G$, define $b=G^{1/(2-m)}$. Then, Colding defined two functions, the `area' $A$ and the `volume' $V$, on the level sets of $b$ as,
\be\label{area Colding}
A(s) = \frac{1}{s^{m-1}}\int_{b=s} (|\nabla b|^{2} -1)|\nabla b|\, dA,
\ee
and,
\be\label{volume Colding}
V(s) = \frac{1}{s^{m}}\int_{b\leq s} (|\nabla b|^{2} -1)|\nabla b|^{2}\, dV. 
\ee
%Observe that for $\mathbb{R}^{m}$, $A(s)=0$ and $V(s)=0$. 
Having $A$ and $V$, the monotonic volume is defined as
\be\label{monotonic volume Colding}
W(s)=2(m-1)V(s)-A(s), 
\ee
and the derivative of $W$ is given by the expression (see \cite{C12}, Theorem 2.4),
\be\label{derivative of Coldings volume}
\frac{d}{ds}W(s)= -\frac{1}{2s^{m+1}} \int_{b\leq s} \left|\nabla\nabla b^{2} - \frac{\Delta b^{2}}{m}g\right|^{2} dV,
\ee
from which it follows that it is monotonically decreasing in $s$. For instance, on the manifold $\mathbb{R}^{m}$, the Green function at the origin is given by $G(x) = 1/|x|^{m-2}$. If we define $b = G^{1/(2-m)}$, then $b = |x|$, and therefore, $|\nabla b| = 1$. Consequently, in $\mathbb{R}^{m}$, the area, volume, and monotonic volume are all identically zero.
\vs

We now outline how to derive Perelman's entropy from Colding's monotonic volume. Define  $h: \hat{M} \to \mathbb{R}$ by, 
    \be\label{def h}
    h = r^{-(m-2)}e^{-f(\tau,x)},
    \ee
where $f$ satisfies (\ref{BACKCHO}) and $\tau=r^2/2N$. The function $h$ will serve as the analog to the Green function when applying Colding's monotonic formula to Perelman's $N$-space. The observation that $h$ is almost harmonic is essentially due to Perelman. In Section 6.1 of \cite{Per02}, it is observed that $$\tilde{u}^* := \tau^{-\frac{N-1}{2}}u = (2N)^{\frac{m-2}{2}}h,$$ is harmonic modulo $N^{-1}$ if and only if $f$ satisfies (\ref{BACKCHO}) (the precise statement is $\hat{\Delta}\tilde{u}^* = \tau^{(2-m)/2}O(1/N)$, see Proposition \ref{prop lap h}). From this definition, and following \cite{C12}, we define $b: \hat{M} \to \mathbb{R}$ as,
    \be\label{def b}
    b = h^{1/(2-m)}= re^{f/(m-2)}.
    \ee
    Using this $b$, we define the quantities $\mathcal{A}_N$, $\mathcal{V}_N$, and $\mathcal{W}_N$ as in (\ref{area Colding}), (\ref{volume Colding}), and (\ref{monotonic volume Colding}), respectively, each scaled by a coefficient $c_N$ (see Definition \ref{quantities}), and prove the following main results. Firstly, in Corollary \ref{convergence monot Colding} of Section \ref{sect 3} we demonstrate that the monotonic volume $\mathcal{W}_N$  and its derivative converge uniformly on compact sets in $(0, T)$ to Perelman's $\mathcal{W}$-entropy and its derivative, respectively. Secondly, in Theorem \ref{Final Theo} of Section \ref{sect 4}, we obtain Perelman's formula (\ref{derivative of W}) for the derivative of $\mathcal{W}$ as a limit of the derivative (\ref{derivative of Coldings volume}) in Perelman's $N$-space.  Besides these two results, and as a byproduct of the proofs, in Theorem \ref{WN convergence} we also show that the area $\mathcal{A}_N$ and its derivative also converge uniformly to Perelman's $\mathcal{W}-$entropy and its derivative respectively on compact sets in $(0, T)$.

\vs 

The body of the work is organized as follows: In Section \ref{sect 2}, we introduce the relevant notation, definitions, and frequently referenced tools. In Section \ref{sect 3}, we derive Perelman's entropy as a limit of Colding's volume. Finally, in Section \ref{sect 4}, we demonstrate that Perelman's formula for the derivative of the entropy can be obtained as a limit of Colding's formula in $\hat{M}$.

\vs

%%%%%%%%%%%%%%%%%%%%%%%%%%%%%%%%%%%%%%%%%%%%%%%%%

\section{Preliminaries}\label{sect 2}

Here we compile some results, observations, and definitions that will be frequently referenced throughout the work. We begin with a few remarks on Perelman's space, $(\hat{M}, \hat{g}).$

Strictly speaking, $\hat{g}$ is a metric that depends on $N$, but we will omit the subscript $N$ to simplify notation. In fact, we will omit the subscript $N$ in most parts of the work. For instance, the Ricci curvature of $\hat{g}$ will be denoted by $\hat{\Ric}$, the covariant derivative will be $\hat{\nabla}$, and so on. 

The metric coefficient $(1+Rr^{2}/N^{2})=(1+2\tau R/N)$ will appear often, so to simplify notation we define, 
\be
v := 1+\frac{Rr^{2}}{N^{2}}. 
\ee

Note that $\hat{g}$ is invariant under rotations in $\mathbb{R}^{N+1}$ and therefore, scalar quantities like the scalar curvature $\hat{R}$, or the norm of the Ricci tensor $|\hat{\Ric}|$, are also invariant under rotations and thus $\theta$-independent. For this reason, they will often be considered as functions on $(0,T)\times M$. For instance, $v$ can be viewed as a function on $\hat{M}$ or $(0,T)\times M$, depending on the context.

In \cite{Per02}, Perelman showed that $|\hat{\Ric}| = O(1/N)$, (see \cite{CZ06} for the full computation of the Christoffel symbols and curvature components). This means that the sequence of functions $N|\hat{\Ric}|$, as functions on $(0,T)\times M$, is uniformly bounded on compact sets. 

Regarding the definition of order, we will need the following.

\begin{Definition}
Let $i\geq 0$ be an integer. A sequence of real-valued functions $F_{N}(\tau, \theta,x)$ is an $O_0(1/N^{i})$  if for every $0<\tau_{1}<T$ there exists $K>0$ such that,
\be\nonumber
N^{i}|F_{N}(\tau, \theta,x)|\leq K, 
\ee 
for all $N>0$ and for all $(\tau, \theta, x)$ such that $\tau_{1}\leq \tau\leq T$. 

We say that $F_N(\tau,\theta,x)$ is an $O_k(1/N^i)$   if $\partial^{\alpha}F_N$ is an $O_0(1/N^i)$,  for any multi-index $|\alpha|\leq k$, where the derivatives are taken in the $\tau$ or $x^i$ variables.
\end{Definition}

\begin{Definition}
    We say that a sequence of real-valued functions $F_N(\lambda)$, defined in $(0,T)$, is an $O_0(e^{-cN})$ if for each $\lambda_0 \in (0,T)$, given $\lambda_1>\lambda_0$, there exist $K>0$ and $c>0$ such that, 
    \be
    e^{cN}|F_{N}(\lambda)|\leq K,
    \ee
    for all $N>0$ and for all $\lambda_1 \leq \lambda \leq T.$
     We say that $F_N(\lambda)$ is an $O_k(e^{-cN})$ if $d^jF_N/d\lambda^j$ is an $O_0(e^{-cN})$ for every $j \leq k$.
\end{Definition}

A direct application of the Taylor expansion with the remainder in integral form allows us to derive the following lemma, which we will frequently reference.

\begin{Lemma}\label{lemma order} Let $F:(a,b)\rightarrow \mathbb{R}$ be a smooth real-valued function. Let $w(\tau,x,N)$ be a smooth real-valued function with range in $(a,b)$, such that,
\be\nonumber
w(\tau,x,N) = w_{0}(\tau,x) + \delta(\tau,x,N),
\ee
where,
\be\nonumber
\delta(\tau,x,N) = O_{k}(\frac{1}{N^{j}}),
\ee
for some integers $k\geq 0$ and $j\geq 0$. Then, for any $l\geq 1$, $F(w(\tau,x,N))$ has the following decomposition $($where we omit the $\tau,x$ and $N$ dependence for notational convenience$)$,
\be\nonumber
F(w) = F(w_{0}) + F'(w_{0})\delta + F''(w_{0})\frac{\delta^{2}}{2}+\ldots+F^{(l-1)}(w_{0})\frac{\delta^{l-1}}{(l-1)!} + R_{l},
\ee
and,
\be\nonumber
R_{l}(\tau,x,N) = O_{k}(\frac{1}{N^{jl}}).
\ee
\end{Lemma}
Note that the usual rules of orders hold: the sum  $O_k(1/N^{i})+O_l(1/N^{j})$ is an $O_{\min {(k,l)}}(1/N^{\min{(i,j)}})$ and the product $O_k(1/N^{i})O_l(1/N^{j})$ is an $O_{\min {(k,l)}}(1/N^{i+j})$. For instance, in Lemma \ref{lemma order}, if $\delta=O_k(1/N^j),$ then $\delta^{i}=O_k(1/N^{ji}).$ 

As an example of Lemma \ref{lemma order}, $$v=1+ 2\tau R/N= 1+\delta,  \text{\, where \,} \delta= 2\tau R/N =  O_k(1/N),$$ and  therefore, $$1/\sqrt{v} =1-\tau R/N + O_2(\frac{1}{N^2}).$$ 
We observe that, while as in this and other cases,  $\delta= O_k(1/N)$ for any $k\geq 0$, for our purposes it will be enough to use $k \leq 2$. Similarly, $$e^{f/(m-2)} = 1+ \frac{f}{N}+ O_2(\frac{1}{N^2}) = 1 + O_2(\frac{1}{N}).$$
We will not delve into detail for most order computations, as they primarily involve combinations of compositions and products of the previous examples, along with straightforward applications of Lemma \ref{lemma order}.

Now we proceed to make some key observations about the level sets of $b$.  Observe that, since $b/\sqrt{2N} = \sqrt{\tau} e^{f/(m-2)}$, for every $k\geq 0$ we have,
\be\label{FOL}
\frac{b}{\sqrt{2N}} = \sqrt{\tau} +O_{k}(\frac{1}{N}).
\ee
Therefore, for every small $\delta>0$ and $k\geq 0$, the sequence of functions $$b/\sqrt{2N}: [\delta,T-\delta]_{\tau}\times M_{x} \rightarrow \mathbb{R}$$ converges in $C^{k}$ to the function (coordinate) $\tau:[\delta,T-\delta]\times M\rightarrow \mathbb{R}$. By standard calculus, it follows that there exists $N_{0}>0$ such that for  every $N > N_{0}$ and $\lambda \in [2\delta,T-2\delta]$, the level set $b/\sqrt{2N} = \sqrt{\lambda}$ is given by an immersion
\be\label{implicit map}
x\in M\mapsto (\phi_{N,\lambda}(x),x)\in (0,T)\times M,
\ee
for some smooth function $\phi_{N,\lambda}:M\rightarrow \mathbb{R}$. Furthermore, from (\ref{FOL}) we get,
\be\label{order implicit map}
\lambda = \frac{b^2}{2N}(\phi_{N,\lambda}(x),x) = \tau (\phi_{N,\lambda}(x),x) + O_k(\frac{1}{N}) = \phi_{N,\lambda}(x) + O_{k}(\frac{1}{N}),
\ee
that is, the immersion approaches the level set $\tau=\lambda$.
Since we will be interested in taking limits, we will always assume that $N$ is sufficiently big such that (\ref{implicit map}) holds inside a region $\lambda \in [\delta, T-\delta]$, where $\delta$ is sufficiently small. 

\vs
We conclude the section by defining $\mathcal{A}_N,$ $\mathcal{V}_N$ and $\mathcal{W}_N$.

\begin{Definition}\label{quantities}
    Let $$c_N := (4\pi)^{-n/2}(2N)^{n/2+1}/4|\mathbb{S}^N|.$$
    We define the area $\mathcal{A}_N:(0,\sqrt{2NT}) \to \mathbb{R} $ $($resp. the raw area  $\overline{\mathcal{A}}_N:(0,\sqrt{2NT}) \to \mathbb{R} $$)$ as,
    \be\label{def A_N}
    \mathcal{A}_N(s)= \frac{c_N}{s^{m-1}} \int_{b=s} (|\hat{\nabla}b|^2-1)|\hat{\nabla{b}}|\, d\hat{A}, \,\, \left( \mbox{resp. } \overline{\mathcal{A}}_N(s)= \frac{c_N}{s^{m-1}} \int_{b=s} |\hat{\nabla}b|^3\, d\hat{A} \right),
    \ee
    the volume $\mathcal{V}_N :(0,\sqrt{2NT}) \to \mathbb{R} $ as,
    \be\label{def V_N}
    \mathcal{V}_N (s)= \frac{c_N}{s^m}\int_{ b \leq s} (|\hat{\nabla}b|^2-1)|\hat{\nabla{b}}|^2\, d\hat{V},
    \ee
    and the monotonic volume $\mathcal{W}_N:(0,\sqrt{2NT}) \to \mathbb{R} $ as,
    \be\label{def W_N}
    \mathcal{W}_N (s) =   (2(m-1) \mathcal{V}_N  - \mathcal{A}_N)(s).
    \ee
\end{Definition}

\begin{Remark}\label{Remark} We always associate $\lambda \in(0,T)$ with $s=\sqrt{2N\lambda} \in (0,\sqrt{2NT})$.
    Quantities that are functions of $s$, for instance $\mathcal{A}_{N}(s)$, will be considered sometimes as functions of $\lambda \in (0,T)$, in which case we will write, for example, $\mathcal{A}_{N}(\lambda)$ instead of $\mathcal{A}_{N}(s=\sqrt{2N\lambda})$. Convergence of $\mathcal{A}_{N}$ or  $\mathcal{V}_{N}$ to the entropy $\mathcal{W}$ occurs when considering them as functions of $\lambda \in (0,T)$.
\end{Remark}

%%%%%%%%%%%%%%%%%%%%%%%%%%%%%%%%%%%%%%%%%%%%%%%%%%%%%%%%

%%%%%%%%%%%%%%%%%%%%%%%%%%%%%%%%%%%%%%%%%%%%%%%%%%%%%%%%

\section{Derivation of Perelman's Entropy from Colding's Monotonic Volume}\label{sect 3}

We will make repeated use of the following results.
\begin{Lemma}
Let $f = f(\tau,x)$ be any smooth real-valued function, that we consider as a function on $\hat{M}$. Then,
\be\label{expression for grad b2}
|\hat{\nabla} b|^{2} = 1 + \frac{1}{N}\left( 2f - 2\tau R + 2\tau |\nabla f|^{2} + 4\tau \partial_{\tau} f\right) +O_2(\frac{1}{N^{2}}).
\ee
In particular, if $f$ satisfies $($\ref{BACKCHO}$)$, then,
\be\label{ZERRO}
\frac{N}{2}(|\hat{\nabla}b|^{2} - 1) = \tau(2\Delta f - |\nabla f|^{2} + R) + f - n + O_2(\frac{1}{N}),
\ee
and therefore,
\be\label{GRADORD}
|\hat{\nabla}b| = 1 + O_2(\frac{1}{N}).
\ee

\begin{proof}
We compute,
\be\label{grad b}
\hat{\nabla} b = \frac{e^{f/(m-2)}}{\sqrt{v}}\left(1+\frac{r\partial_{r}f}{m-2}\right)\frac{\partial_{r}}{\sqrt{v}} + \frac{re^{f/(m-2)}}{m-2}\nabla f.
\ee
Thus,
\be\label{SR}
|\hat{\nabla} b|^{2} = \frac{e^{2f/(m-2)}}{v}\left(1+\frac{r\partial_{r}f}{m-2}\right)^{2} + \frac{r^{2}e^{2f/(m-2)}}{(m-2)^{2}}|\nabla f|^{2}.
\ee

Now, $r\partial_{r}f = 2\tau\partial_{\tau}f$. Therefore,
\be
\left(1+\frac{r\partial_{r}f}{m-2}\right)^{2} = 1 + \frac{4\tau \partial_{\tau} f}{N} + O_2(\frac{1}{N^{2}}).
\ee
Also, by virtue of Lemma \ref{lemma order} (see also the discussion below the Lemma),
\be\nonumber
\frac{e^{2f/(m-2)}}{v} = \left(1+\frac{2f}{m-2} + O_2(\frac{1}{N^{2}})\right) \left(1 - \frac{2\tau R}{N} + O_2(\frac{1}{N^{2}})\right) = 1+\frac{2f}{N} - \frac{2\tau R}{N} + O_2(\frac{1}{N^{2}}),
\ee
and similarly,
\be\nonumber
\frac{r^{2}e^{2f/(m-2)}}{(m-2)^{2}}|\nabla f|^{2} = \frac{2N\tau}{(m-2)^{2}} e^{2f/(m-2)}|\nabla f|^{2} = \frac{2\tau |\nabla f|^{2}}{N} + O_2(\frac{1}{N^{2}}).
\ee
After putting altogether inside (\ref{SR}), we obtain (\ref{expression for grad b2}).
Finally, we can replace $\partial_{\tau}f$ by (\ref{BACKCHO}) to get (\ref{ZERRO}).
\end{proof}

\end{Lemma}

%%%%%%%%%%%%%%%%%%%%%%%%%%%%%%%%%%%%%%%%%%%%%%%%

%%%%%%%%%%%%%%%%%%%%%%%%%%%%%%%%%%

\begin{Proposition}\label{Vol Element}
 For any $\lambda \in (0,T)$, the volume element for the level set $b=s=\sqrt{2N\lambda}$ can be expressed as,
    \be\label{Vol Elm in s}
    d\hat{A}= s^N e^{-f(\lambda, x)}(1+O_2(\frac{1}{N}))d\nu d\nu_{\mathbb{S}^N}, 
    \ee
    where $d\nu_{\mathbb{S}^N}$ is the standard volume element in $\mathbb{S}^N$, and $d\nu$ is the volume element in $(M,g(\lambda))$.
    In particular, we have
    \be\label{Vol elm in lambda}
    d\hat{A}= (2N\lambda)^{N/2} e^{-f(\lambda,x)}(1+O_2(\frac{1}{N}))d\nu d\nu_{\mathbb{S}^N}. 
    \ee
    \begin{proof}
       Since $\lambda \in (0,T),$ there exists $\delta>0$ such that $\lambda \in [\delta,T-\delta]$. Then, using (\ref{implicit map}),
        for every $N>N_0$ we define a map $\psi_{N,\lambda} :  \mathbb{S}^N \times M \to \hat{M}$ as
        \be 
        \psi_{N,\lambda}(\theta, x) = (\sqrt{2N\phi_{N,\lambda}}(x), \theta, x).
        \ee
     Let $\theta^{\alpha}$ be coordinates in $\mathbb{S}^{N}$ and $x^{i}$ coordinates on $M$. Since $r=\sqrt{2N\phi_{N,\lambda}}(x)$ and by (\ref{order implicit map}) we have $\phi_{N,\lambda}=\lambda+O_k(N^{-1})$,  we compute, 
    \begin{align*}
        \psi_{N,\lambda}^* \hat{g} (\partial_i, \partial_j) &  = \hat{g}(d\psi_{N,\lambda} (\partial_i),d\psi_{N,\lambda}(\partial_j)) \\ &= \hat{g}\left(\frac{\sqrt{N}}{\sqrt{2\phi_{N,\lambda} }}\partial_i\phi_{N,\lambda}\partial_r+\partial_i, \frac{\sqrt{N}}{\sqrt{2\phi_{N,\lambda} }}\partial_j\phi_{N,\lambda}\partial_r+\partial_j\right) \\ & = g_{ij} + O_2(\frac{1}{N}),\\
       \psi_{N,\lambda}^* \hat{g} (\partial_{\alpha}, \partial_{\beta}) & = \hat{g}(\partial_\alpha, \partial_\beta) = (2N\phi_{N,\lambda})g_{\mathbb{S}^N \alpha \beta}, \\
        \psi_{N,\lambda}^* \hat{g} (\partial_{\alpha}, \partial_{i})& =0,
    \end{align*}
    for every $N>N_0$.
    Therefore, we can express the volume element $d\hat{A}$ as,
    \be\nonumber
    d\hat{A}= (2N\phi_{N,\lambda})^{N/2}(1+O_2(\frac{1}{N}))d\nu d\nu_{\mathbb{S}^N},
    \ee
    where we are using Lemma \ref{lemma order} when computing $\sqrt{\psi_{N,s}^* \hat{g}}$. 
    Since $r^2/2N=\tau= \phi_{N,\lambda}(x)$ for any  $(r,\theta,x)\in b=s$, we have,
    \be\label{expansion r}
    r=\sqrt{2N\phi_{N,\lambda}}(x)  = se^{f(\phi_{N,\lambda}(x),x)/(2-m)},
    \ee
    and by (\ref{order implicit map}) and Lemma \ref{lemma order} we can write,
    \be\nonumber
    \begin{split}
    (2N\phi_{N,\lambda})^{N/2} &= s^N e^{-f(\lambda+O_k(1/N),x)+O_2(1/N)}\\
    &=s^N e^{-f(\lambda ,x)}(1+O_2(\frac{1}{N})),
    \end{split}
    \ee
    which completes the proof.
    \end{proof}
\end{Proposition}

%%%%%%%%%%%%%%%%%%%%%%%%%%%%%%%%%%%%%%%%%%%%%%%%
We now demonstrate that Perelman's $\mathcal{W}$-entropy and its derivative emerge as the limits of the area $\mathcal{A}_N$ and its derivative, respectively, when considered as functions of $\lambda$ (see Remark \ref{Remark}).
\begin{Theorem}\label{WN convergence}
The following equality holds:
\be\label{W=A_N+O}
\mathcal{A}_N (\lambda)=\mathcal{W} (\lambda)+ O_2(\frac{1}{N}),
\ee
where $\mathcal{W}$ is Perelman's entropy functional $($\ref{Def W}$)$.
In particular,
\be
\mathcal{A}_N \to \mathcal{W},
\ee
and   
\be
\frac{d\mathcal{A}_N}{d\lambda} \to \frac{d\mathcal{W}}{d\lambda},
\ee
uniformly on compact sets in $(0,T)$.
\begin{proof}
Let $s=\sqrt{2N\lambda}$. We compute,
\begin{align*}
\mathcal{A}_N & =\frac{c_N}{s^{m-1}}\int_{b=s} (|\hat{\nabla}b|^2-1)|\hat{\nabla b}| d\hat{A} \\
& = (4\pi\lambda)^{-n/2}\int_M \left(\lambda(2\Delta f - |\nabla f|^{2} + R) + f -n \right)e^{-f(\lambda,x)}(1+O_2(\frac{1}{N}))d\nu\\
& = (4\pi\lambda)^{-n/2}\int_M \left(\lambda(2\Delta f - |\nabla f|^{2} + R) + f -n\right)e^{-f(\lambda,x)}d\nu + O_2(\frac{1}{N}),
\end{align*}
where we used (\ref{ZERRO})  and  (\ref{Vol elm in lambda}) together with the fact that the integrand does not depend on $\theta \in \mathbb{S}^N$. Since $M$ is a closed manifold,
\be\nonumber
\int_M (\Delta f - |\nabla f|^2)e^{-f}d\nu = -\int_M \Delta e^{-f}d\nu=0,
\ee
and we use it to rewrite, 
\be\nonumber
\mathcal{A}_N= (4\pi\lambda)^{-n/2}\int_M \left(\lambda(|\nabla f|^2+R)+f-n\right)e^{-f}d\nu + O_2(\frac{1}{N}),
\ee
from which (\ref{W=A_N+O}) follows. The previous expression also shows that $\mathcal{A}_N \to \mathcal{W}$ uniformly on compact sets in $(0,T)$, and differentiating with respect to $\lambda$ on both sides of (\ref{W=A_N+O}),
\be\label{Der WN}
\frac{d\mathcal{A}_N }{d\lambda}= \frac{d\mathcal{W}}{d\lambda} + O_1(\frac{1}{N}),
\ee
 which also shows that convergence of the derivatives is uniform on compact sets in $(0,T)$.
    
\end{proof}
\end{Theorem}

We now proceed to study the convergence of the volume. To begin, we introduce a lemma.
 
%%%%%%%%%%%%%%%%%%%%%%%%%%%%%%%%%%%%%%%%%%%%%

\begin{Lemma}\label{Lemma vol}
Let $\lambda_0 <\lambda \in (0, T)$. Define $\bar{s}=\sqrt{2N\lambda_0}$ and  $s=\sqrt{2N\lambda}$. Then,
    \be\label{decaying term in V_N}
        \mathcal{V}_N(s) = \frac{c_N}{s^m}\int_{\bar{s} \leq b \leq s} (|\hat{\nabla}b|^2-1)|\hat{\nabla}b|^2 \,d\hat{V} + O_1(e^{-cN}).
        \ee
    \begin{proof}
        We write, 
        \be\nonumber
        \mathcal{V}_N(s) = \frac{c_N}{s^m}\int_{\bar{s} \leq b \leq s} (|\hat{\nabla}b|^2-1)|\hat{\nabla}b|^2 \, d\hat{V} +\frac{c_N}{s^m}\int_{0 \leq b \leq \bar{s}} (|\hat{\nabla}b|^2-1)|\hat{\nabla}b|^2 \, d\hat{V}.
        \ee
        Since $f=-\ln u -(n/2)\ln \tau,$ a straightforward computation using (\ref{grad b}) shows that we can write $|\hat{\nabla} b|^2$ in terms of $u$ as,
        \be\nonumber
        \begin{split}
        |\hat{\nabla} b|^2  &=(2N)^{n/(m-2)}(ur^n)^{2/(2-m)}\left[ \left(1 - \frac{2\tau \partial_\tau u -nu}{u(m-2)}\right)^2 + \frac{r^2}{(m-2)^2}\frac{|\nabla u|^2}{u^2}\right]\\
    &= (2N)^{n/(m-2)}(ur^n)^{2/(2-m)}H_N(r^2/2N,x),
        \end{split}
        \ee  
        and, since $u$ is everywhere positive in $[0,\lambda_0] \times M$ and its first derivatives are bounded, it follows that $|H_N| \leq C$  and $0<c_1<u<c_2$ in $[0,\lambda_0] \times M$, for some constants $C,c_1,c_2>0$. In particular, the same bounds hold for the region $0 \leq b \leq \bar{s}$, and since $d\hat{V}=r^{N}drd\nu d\nu_{\mathbb{S}^N},$ we use the  previous expression for $|\hat{\nabla} b|^2$ to find,
        \be\label{term in 0 bar s}
        \left| \frac{c_N}{s^m}\int_{0 \leq b \leq \bar{s}}|\hat{\nabla}b|^4 d\hat{V} \right| \leq K_1\left(\frac{\bar{s}}{s}\right)^{N+1},
        \ee
        for some constant $K_1>0$. Since $\bar{s}/s = \lambda_0/\lambda <1$, it decays exponentially fast.  Similarly, 
        \be\label{term in norm of Vn bar s}
        \left| \frac{c_N}{s^m}\int_{0 \leq b \leq \bar{s}}|\hat{\nabla}b|^2 d\hat{V} \right| \leq K_2\left(\frac{\bar{s}}{s}\right)^{N+1},
        \ee
        for some constant $K_2>0$, which shows
        \be\nonumber
        \mathcal{V}_N = \frac{c_N}{s^m}\int_{\bar{s} \leq b \leq s} (|\hat{\nabla}b|^2-1)|\hat{\nabla}b|^2 \,d\hat{V} + O_0(e^{-cN}).
        \ee
        Finally, in order to see that the term $O_0(e^{-cN})$ is in fact $O_1(e^{-cN})$, we note that
        \be\nonumber
        \begin{split}
        \frac{d}{d\lambda}\left(\frac{c_N}{s^m}\int_{0 \leq b \leq \bar{s}}(|\hat{\nabla}b|^2-1)|\hat{\nabla}b|^2 \,d\hat{V} \right) &= \frac{ds}{d\lambda}\frac{d}{ds}\left(\frac{c_N}{s^m}\int_{\bar{0} \leq b \leq \bar{s}}(|\hat{\nabla}b|^2-1)|\hat{\nabla}b|^2 \,d\hat{V} \right) \\
        &=-\left(\frac{N}{2\lambda}\right)^{1/2}\frac{m}{s}\frac{c_N}{s^m}\int_{\bar{0} \leq b \leq \bar{s}}(|\hat{\nabla}b|^2-1)|\hat{\nabla}b|^2 d\hat{V} ,
        \end{split}
        \ee
        and apply (\ref{term in 0 bar s}) and (\ref{term in norm of Vn bar s}) to see that its derivative is also an $O_0(e^{-cN})$.
    \end{proof}
\end{Lemma}

%%%%%%%%%%%%%%%%%%%%%%%%%%%%%%%%%%%%%%%%%%%%%%%

\begin{Proposition}\label{derivative Colding}
The following equality holds,
    \be\label{der monot colding}
   2(m-1) \mathcal{V}_N  - \mathcal{A}_N= \mathcal{A}_N + O_1(\frac{1}{N}).
    \ee
    In particular, 
    \be\label{derivative non normalized monotonicity Colding}
    \frac{d}{ds}(2(m-1) \mathcal{V}_N  - \mathcal{A}_N) =  \frac{d}{ds}\mathcal{A}_N + O_0(\frac{1}{N^{3/2}}).
    \ee
    \begin{proof}
        Let $\lambda \in (0, T)$, and define $s = \sqrt{2N\lambda}$. Next, we choose $\lambda_0 < \lambda$ and define $\bar{s} = \sqrt{2N\lambda_0}$. By the coarea formula and equation (\ref{decaying term in V_N}) we obtain,
        \be\nonumber
        \mathcal{V}_N = \frac{c_N}{s^m}\int_{\bar{s} \leq b \leq s} (|\hat{\nabla}b|^2-1)|\hat{\nabla}b|^2 d\hat{V} +O_1(e^{-cN}) = \frac{1}{s^m}\int_{\bar{s}}^s w^{m-1} \mathcal{A}_N(w) dw +O_1(e^{-cN}).
        \ee
    Using (\ref{Der WN}), we deduce,
        \be\label{dA_N/ds}
         \begin{split}
    \frac{d\mathcal{A}_N}{ds}&= \frac{d\lambda}{ds}\frac{d\mathcal{A}_N}{d\lambda}\\& =(4\pi)^{n/2}\frac{d\lambda}{ds}\left( \frac{d\mathcal{W}}{d\lambda} + \frac{d}{d\lambda}\left(O_2(\frac{1}{N})\right)\right)\\
    & = (4\pi)^{n/2}\left(\frac{2\lambda}{N}\right)^{1/2}\left(\frac{d\mathcal{W}}{d\lambda} + O_1(\frac{1}{N})\right)\\
    & = O_1(\frac{1}{N^{1/2}}),
    \end{split}
    \ee
    and,
    \be\label{second derivative A_N}
    \frac{d^2\mathcal{A}_N}{ds^2} = \frac{d\lambda}{ds}\frac{d}{d\lambda}\left( O_1(\frac{1}{N^{1/2}})\right) = O_0(\frac{1}{N}),
    \ee
    since  $d\mathcal{W}/d\lambda = O_2(1)$. 
         A Taylor expansion around $s$ shows,
        \be\nonumber
        \mathcal{A}_N(w)=\mathcal{A}_N(s) + (w-s)\frac{d\mathcal{A}_N}{ds}(\xi_{w,s}),
        \ee
        for some $\sqrt{2N\lambda_0} \leq \xi_{w,s} \leq \sqrt{2N\lambda}$. The bound on $\xi_{w,s}$ implies the respective bound for its associated $\tau-$coordinate between $\lambda_0$ and $\lambda$, and therefore by (\ref{dA_N/ds}),
        \be\nonumber
        \mathcal{A}_N(w)=\mathcal{A}_N(s) + (w-s)O_0(\frac{1}{N^{1/2}}),
        \ee
        for any $w\in [\bar{s},s]$. Now, we notice that, since $|O_0(N^{-1/2})|\leq CN^{-1/2}$,
        \be\label{Argument for o0 taylor}
        \begin{split}
         -\frac{s^{m+1}}{s^m(m+1)m}\frac{C}{N^{1/2}} + O_0(e^{-cN}) &\leq
        \frac{1}{s^m}\int_{\bar{s}}^s w^{m-1} \left((w-s)O_0(\frac{1}{N^{1/2}})\right) dw \\
         &\leq \frac{s^{m+1}}{s^m(m+1)m}\frac{C}{N^{1/2}} + O_0(e^{-cN}),
        \end{split}
        \ee
        which shows, 
        \be
         \frac{1}{s^m}\int_{\bar{s}}^s w^{m-1} \left((w-s)O_0(\frac{1}{N^{1/2}})\right) dw = O_0(\frac{1}{N^2}).
        \ee
        Then,
        \be\label{Taylor for V_N}
        \begin{split}
        \mathcal{V}_N& = \frac{1}{s^m}\int_{\bar{s}}^s w^{m-1} \left(\mathcal{A}_N(s)+(w-s)O_0(\frac{1}{N^{1/2}})\right) dw  +O_1(e^{-cN})\\
        & = \frac{1}{m}\mathcal{A}_N + O_0({\frac{1}{N^2}}).
        \end{split}
        \ee
        In order to see that the term $O_0(N^{-2})$ is actually an $O_1(N^{-2})$, we use a Taylor expansion of order $2$ to compute,
        \be\label{derivative for Taylor V_N}
        \begin{split}
        \frac{d\mathcal{V}_N}{ds} & = \frac{d}{ds}\left(\frac{1}{s^m}\int_{\bar{s}}^s w^{m-1} \mathcal{A}_N(w) dw\right) +O_0(e^{-cN}) \\
        & = -\frac{m}{s^{m+1}}\int_{\bar{s}}^s w^{m-1} \mathcal{A}_N(w) dw +  \frac{1}{s}\mathcal{A}_N(s)  +O_0(e^{-cN}) \\
    & = -\frac{m}{s^{m+1}}\int_{\bar{s}}^s w^{m-1} \left( \mathcal{A}_N(s) + (w-s)\frac{d}{ds}\mathcal{A}_N(s)+ \frac{(w-s)^2}{2}\frac{d^2\mathcal{A}_N}{ds^2}(\xi_{w,s}) \right)dw \\
    & \phantom{....}+  \frac{1}{s}\mathcal{A}_N(s)  +O_0(e^{-cN}),
        \end{split}
        \ee
    where $\xi_{w,s} \in [w,s].$ Now, since $d^2\mathcal{A}_N/ds^2 = O_0(1/N),$ we proceed as in $(\ref{Argument for o0 taylor})$ and show,
    \be
    \begin{split}
        \int_{\bar{s}}^s w^{m-1} &\left( \mathcal{A}_N(s) + (w-s)\frac{d\mathcal{A}_N}{ds}(s)+ \frac{(w-s)^2}{2}\frac{d^2\mathcal{A}_N}{ds^2}(\xi_{w,s}) \right)dw \\
        & = \frac{s^m}{m}\mathcal{A}_N(s) - \frac{s^{m+1}}{m(m+1)}\frac{d\mathcal{A}_N}{ds}(s) + \frac{s^{m+2}}{m(m+1)(m+2)}O_0(\frac{1}{N}) + O_0(e^{-cN}).
    \end{split}
    \ee
    Use this in (\ref{derivative for Taylor V_N}) and the fact that $d\mathcal{A}_N/ds = O_1(N^{-1/2})$ to write,
    \be
    \frac{d\mathcal{V}_N}{ds} = \frac{1}{m+1}\frac{d\mathcal{A}_N}{ds}(s) + O_0(\frac{1}{N^{5/2}}) = \frac{1}{m}\frac{d\mathcal{A}_N}{ds}(s) + O_0(\frac{1}{N^{5/2}}).
    \ee
     This implies,  
        \be\nonumber
        \frac{d\mathcal{V}_N}{d\lambda}=\frac{ds}{d\lambda}\frac{d\mathcal{V}_N}{ds} = \frac{1}{m}\frac{d\mathcal{A}_N}{d\lambda} + \frac{ds}{d\lambda}O_0(\frac{1}{N^{5/2}})=\frac{1}{m}\frac{d\mathcal{A}_N}{d\lambda} + O_0(\frac{1}{N^{2}}),
        \ee 
        which we combine with (\ref{Taylor for V_N}) to obtain,
        \be
        \mathcal{V}_N = \frac{1}{m}\mathcal{A}_N + O_1(\frac{1}{N^2}).
        \ee
        
        From here, a straightforward computation proves (\ref{der monot colding}). We finish the proof by differentiating (\ref{der monot colding}) with respect to $s$ and noting that
        \be\nonumber
        \frac{d}{ds}\left(O_1(\frac{1}{N})\right) = \frac{d\lambda}{ds}\frac{d}{d\lambda}\left(O_1(\frac{1}{N})\right) = O_0(\frac{1}{N^{3/2}}).
        \ee
    \end{proof}
\end{Proposition}
%%%%%%%%%%%%%%%%%%%%%%%%%%%%%%%%%%%%%%%%%%%%%%%%%
As a consequence of Theorem \ref{WN convergence} and Proposition \ref{derivative Colding}, 
 recalling the definition of the  monotonic volume (\ref{def W_N}) we immediately derive the following result.
\begin{Corollary}\label{convergence monot Colding}
The monotonic volume obeys,
    \be
    \mathcal{W}_N (\lambda) = \mathcal{W} (\lambda) + O_1(\frac{1}{N}),
    \ee
    and,
    \be
    \frac{d}{d\lambda}\mathcal{W}_N(\lambda) = \frac{d}{d\lambda}\mathcal{W}(\lambda)   + O_0(\frac{1}{N}).
    \ee
    In particular, $\mathcal{W}_N \to \mathcal{W}$ and $d\mathcal{W}_N/d\lambda \to d\mathcal{W}/d\lambda$  uniformly on compact sets in $(0,T)$.
\end{Corollary}
%%%%%%%%%%%%%%%%%%%%%%%%%%%%%%%%%%%%%%%%%%%%%%%%%
\section{Computation of the derivative of the entropy from the derivative of the monotonic volume}\label{sect 4}

We now derive Perelman's formula for the derivative of $\mathcal{W}$ as a limit of derivatives of $\mathcal{W}_N$. To achieve this, several auxiliary results are in order.

\begin{Proposition}
Let $h = h(\tau,x)$ be any smooth real-valued function, that we consider as a function on $\hat{M}$. Then,
\begin{align}
\label{LAP} \hat{\Delta} h = & \frac{1}{v}\left(\left(1+\frac{1+2\tau R}{N}-\frac{2(\tau R+\tau^{2} \partial_{\tau}R)}{N^{2}v}\right)\partial_{\tau}h+\frac{2\tau}{N}\partial^{2}_{\tau}h\right)\\
\label{LAPP} & + \Delta h +\frac{1}{Nv}\langle \nabla R,\nabla h\rangle.
\end{align}
\end{Proposition}

\begin{proof}
Let $\theta^{\alpha}$ be coordinates in $\mathbb{S}^{N}$ and $x^{i}$ coordinates on $M$. Then, the inverse metric components of $\hat{g}$ are,
\be\label{gCOMP}
\hat{g}^{ri} = \hat{g}^{r\alpha}=\hat{g}^{\alpha i} =0,\quad \hat{g} = \frac{1}{v},\quad \hat{g}^{ij} = g^{ij},\quad \hat{g}^{\alpha\beta} = \frac{1}{r^{2}}g^{\alpha\beta}_{\mathbb{S}^{N}},
\ee
and the Laplacian therefore is,
\be\label{FEQ}
\hat{\Delta} h = \frac{1}{\sqrt{\hat{g}}}\partial_{r}(\sqrt{\hat{g}}\hat{g}^{rr}\partial_{r} h) + \frac{1}{\sqrt{\hat{g}}}\partial_{i}(\sqrt{\hat{g}}\hat{g}^{ij}\partial_{j}h)+\frac{1}{\sqrt{\hat{g}}}\partial_{\alpha}(\sqrt{\hat{g}}\hat{g}^{\alpha\beta}\partial_{\beta} h).
\ee
A straightforward computation shows that the first term of the r.h.s of (\ref{FEQ}) is equal to the r.h.s of (\ref{LAP}), the second term is equal to the two summands in (\ref{LAPP}) and the last term is equal to zero.
\end{proof}

%%%%%%%%%%%%%%%%%%%%%%%%%%%%%%%%%%%%%%%%%%%%%%%%%%

The following assertion is drawn from Perelman (as mentioned in p.13 of \cite{Per02}), where we needed to include factor of $r^{2-m}$.

\begin{Proposition}[Perelman, \cite{Per02}]\label{prop lap h} Let $f = f(\tau, x)$ be any smooth real-valued function that we consider as a function on $\hat{M}$, and let $h = r^{2-m}e^{-f}$. Then,
\begin{align}\label{GREEN}
\hat{\Delta} h %& = \tau^{-N/2}(-\partial_{\tau} u +\Delta u - Ru+\frac{u}{2\tau})+\tau^{-m/2}O(\frac{1}{N}) \\
%&
 = r^{2-m}\left(\partial_{\tau} f - \Delta f +|\nabla f|^{2} - R +\frac{n}{2\tau}\right)e^{-f}+r^{2-m}O_2(\frac{1}{N}).
\end{align}
Therefore, if $f$ satisfies, $($\ref{BACKCHO}$)$, then
\be\label{DhO}
\hat{\Delta} h = r^{2-m}O_2(\frac{1}{N}).
\ee
\end{Proposition}

\begin{proof}
Let $u = \tau^{-(n-1)/2}e^{-f}$ so that $h =(2N)^{(2-m)/2} \tau^{-N/2}u$. To simplify notation,  we will disregard the multiplicative factor $(2N)^{(2-m)/2}$ in $h$ during the computations, and add it back in afterwards. We compute,
\begin{gather}
\label{PARD}
\partial_{\tau} h = -\frac{N}{2}\tau^{-N/2-1}u+\tau^{-N/2}\partial_{\tau} u,\\
\partial^{2}_{\tau} h = \frac{N}{2}\left(\frac{N}{2}+1\right)\tau^{-N/2-2} u - N \tau^{-N/2-1} \partial_{\tau} u + \tau^{-N/2}\partial^{2}_{\tau} u.
\end{gather} 
We can write $\partial^{2}_{\tau} u = \tau^{-(n+1)/2}O_2(1)$. Therefore,
\be\label{T1}
\frac{2\tau}{N}\partial^{2}_{\tau} h = \tau^{-N/2}\left(\frac{N}{2\tau} u + \frac{1}{\tau} u - 2\partial_{\tau}u\right) + \tau^{-m/2}O_2(\frac{1}{N}).
\ee
Similarly, writing (\ref{PARD}) as,
\be
\partial_{\tau} h = \tau^{-N/2}\left(-\frac{N}{2\tau}u+\partial_{\tau} u\right),
\ee 
we obtain, 
\begin{align}\label{T2}
&\left(1 + \frac{1+2\tau R}{N} -\frac{2(\tau R+\tau^{2} \partial_{\tau}R)}{N^{2}v}\right) \partial_{\tau} h = \\ & \hspace{3cm} - \tau^{-N/2}\left( \frac{N}{2\tau} u + (\frac{1}{2\tau}+R) u - \partial_{\tau} u\right) + \tau^{-m/2}O_2(\frac{1}{N}).
\end{align}
Summing (\ref{T1}) and (\ref{T2}), and after the crucial mutual cancellation of the terms $\tau^{-N/2-1}Nu/2$, we deduce,
\begin{align}
\frac{2\tau}{N}\partial^{2}_{\tau} h & + \left(1 + \frac{1+2\tau R}{N} -\frac{2(\tau R+\tau^{2} \partial_{\tau}R)}{N^{2}v}\right) \partial_{\tau} h = \\
& \hspace{3cm} \tau^{-N/2}(-\frac{1}{2\tau}u - Ru - \partial_{\tau} u) + \tau^{-m/2}O_2(\frac{1}{N}).
\end{align}
Going back to the expression (\ref{LAP}) for $\hat{\Delta} h$, and taking into account that $1/v = 1+O_2(1/N)$ and $\langle \nabla R,\nabla h\rangle/Nv = O_2(1/N)$, we arrive at,
\be
\hat{\Delta} h = \tau^{-N/2}(-\partial_{\tau} u + \Delta u - Ru + \frac{1}{2\tau} u) + \tau^{-m/2}O_2(\frac{1}{N}).
\ee
Recalling that $u = \tau^{-(n-1)/2}e^{-f}$ and multiplying by the factor $(2N)^{(2-m)/2}$, we deduce (\ref{GREEN}).
\end{proof}

%%%%%%%%%%%%%%%%%%%%%%%%%%%%%%%%%%%%%%%%%%%%%%%%%%%%%%%

From the previous proposition we obtain the following.
\begin{Proposition} Let $h$ be defined as in $($\ref{def h}$)$ and $b$ as in $($\ref{def b}$)$. Then,
\be\label{Lap b2}
\hat{\Delta}b^{2} = 2m|\hat{\nabla}b|^{2}+O_2(\frac{1}{N}).
\ee
\end{Proposition}

\begin{proof}
Direct computation shows,
\be\nonumber
\hat{\Delta} b^{2} = 2m|\hat{\nabla} b|^{2} + \frac{2}{2-m}h^{m/(2-m)}\hat{\Delta} h.
\ee 
Using that $\hat{\Delta}h = r^{2-m}O_2(1/N)$ and that $b = h^{1/(2-m)}$, we get,
\begin{align*}
\frac{2}{2-m}h^{m/(2-m)}\hat{\Delta} h = & \frac{2}{2-m} b^{m}r^{2-m}O_2(\frac{1}{N})=\frac{2}{2-m}r^{2}e^{mf/(m-2)}O_2(\frac{1}{N}) =\\
= & \frac{4N}{2-m}\tau e^{mf/(m-2)} O_2(\frac{1}{N}) = O_2(\frac{1}{N}).
\end{align*}
\end{proof}

%%%%%%%%%%%%%%%%%%%%%%%%%%%%%%%%%%%%%%%%%%

The following formula for the derivative of the raw area $\overline{\mathcal{A}}_{N}$  follows essentially from Theorem 2.4 in \cite{C12}, but with two key differences: (i) additional terms appear due to the fact that $\hat{\Delta} h \neq 0$, and (ii) when integrating applying Gauss's theorem, the integration is performed over the region on $\bar{s}\leq b\leq s$, for some $0<\bar{s} < s$, rather than on $b\leq s$.

\begin{Proposition}\label{reformulation theo} The following equality holds:
\be\label{Derivative oAN}
\begin{split}
\frac{d}{ds} \overline{\mathcal{A}}_{N} = & \frac{c_{N}}{2s^{m+1}}\int_{\bar{s}\leq b\leq s}\left(\left|\hat{\nabla}\hat{\nabla} b^{2}-\frac{\hat{\Delta} b^{2}}{m}\hat{g}\right|^{2}+\hat{\Ric}(\hat{\nabla}b^{2},\hat{\nabla} b^{2})\right)d\hat{V} \\ 
& +\frac{2(m-1)}{s}(\overline{\mathcal{A}}_{N}-m\overline{\mathcal{V}}_{N})+ 3[A_{N}] + [B_{N}] + [C_{N}] + [D_{N}],
\end{split}
\ee
where,
\begin{align}
& \overline{\mathcal{V}}_{N} = \frac{c_{N}}{s^{m}}\int_{\bar{s}\leq b\leq s}|\hat{\nabla} b|^{4}\, d\hat{V},\\
& [A_{N}] = \frac{c_{N}}{2-m}\int_{b=s}|\hat{\nabla} b|\hat{\Delta} h\, d\hat{A},\\
& [B_{N}] = \frac{c_{N}}{4 s^{m+1}}\int_{b=\bar{s}}\hat{\nabla}_{\hat{\rm n}} (|\hat{\nabla} b^{2}|^{2})\, d\hat{A},\\
& [C_{N}] = - \frac{c_{N}}{2s^{m+1}}\int_{b=\bar{s}}(\hat{\Delta} b^{2})\hat{\nabla}_{\hat{\rm n}} b^{2}\, d\hat{A},\\
& [D_{N}] = -(1-\frac{1}{m})\frac{c_{N}}{2s^{m+1}}\int_{\bar{s}\leq b\leq s} \left(\frac{8m}{2-m}|\hat{\nabla} b|^{2}b^{m}\hat{\Delta} h +\frac{4 b^{2m}}{(2-m)^{2}}(\hat{\Delta}h)^{2}\right)\, d\hat{V},
\end{align} 
and where $\hat{\rm n}=\hat{\nabla} b/|\hat{\nabla} b|$ is the unit-normal to $b=\bar{s}$. 
\begin{proof}
Follow the proof of Theorem 2.4 in \cite{C12}, but, when computing $(s^{2}\overline{\mathcal{A}}_{N})'/s^{2}$, use that $\hat{\Delta} b^{2} = 2m|\hat{\nabla} b|^{2} +2b^{m}(\hat{\Delta} h)/(2-m)$ instead of $\hat{\Delta} b^{2} =2m|\hat{\nabla} b|^{2}$. Furthermore, when using Gauss's theorem, integrate on $\bar{s}\leq b \leq s$ rather than on $b\leq s$.
\end{proof}

\end{Proposition}

%%%%%%%%%%%%%%%%%%%%%%%%%%%%%%%%%%%%%%%%%%%%%%%%

\begin{Proposition}\label{Der normalizing term}
We have,
\be\label{Derivative of int nabla b}
\frac{d}{ds}\left(\frac{c_N}{s^{m-1}}\int_{b=s}| \hat{\nabla} b| d\hat{A}\right) = [A_N] + O_2(\frac{1}{N^{3/2}}).
\ee
\end{Proposition}

\begin{proof}
Since $\hat{\nabla}_{\hat{n}} b = |\hat{\nabla} b|,$  we compute,
\begin{align}
\frac{d}{ds}\left(\frac{1}{s^{m-1}}\int_{b=s}|\hat{\nabla} b|\, d\hat{A}\right) = & \frac{d}{ds}\left(\int_{b=s}\frac{\hat{\nabla}_{\hat{\rm n}} b}{b^{m-1}}\, d\hat{A}\right)
=  \frac{1}{2-m}\frac{d}{ds}\left(\int_{b=s}\hat{\nabla}_{\hat{\rm n}} h\, d\hat{A}\right) \\ 
\label{DERA} = & \frac{1}{2-m}\int_{b=s} \frac{\hat{\Delta} h}{|\hat{\nabla} b|}\, d\hat{A}.
\end{align}
 Appeal now to (\ref{GRADORD}) to obtain,
\be\label{TR}
\frac{1}{|\hat{\nabla} b|} = |\hat{\nabla} b|\frac{1}{|\hat{\nabla} b|^{2}}=|\hat{\nabla} b| \frac{1}{1 + O_2(\frac{1}{N})} = |\hat{\nabla}b|(1 +O_2(\frac{1}{N})) = |\hat{\nabla}b| +O_2(\frac{1}{N}).
\ee
Then, using (\ref{expansion r}) in (\ref{DhO}), we get $\hat{\Delta}h = s^{2-m}O_2(1/N)$ on the level set $b=s$. Now, combining this, (\ref{TR}) and (\ref{Vol Elm in s}) all in (\ref{DERA}), we deduce,
\begin{align*}
& \frac{d}{ds} \left(\frac{c_N}{s^{m-1}} \int_{b=s}| \hat{\nabla} b| d\hat{A}\right) = \\ 
& \hspace{1cm} = [A_N]  + \frac{c_{N}}{2-m}\int_{M} s^{2-m} O_2(\frac{1}{N}) O_2(\frac{1}{N}) s^{N} (1+O_2(\frac{1}{N}))e^{-f}\, d\nu d\nu_{\mathbb{S}^{N}}.
\end{align*}
But, as $c_{N}=(4\pi)^{-n/2}(2N)^{n/2+1}/(4|\mathbb{S}^{N}|)$ and $s=(2N\lambda)^{1/2}$, we obtain,
\begin{align*}
\frac{c_{N}}{2-m}\int_{M} s^{2-m} O_2(\frac{1}{N}) & O_2(\frac{1}{N}) s^{N} (1+ O_2(\frac{1}{N}))e^{-f} \, d\nu d\nu_{\mathbb{S}^{N}} = \\
& = \frac{(2N)^{n/2+1}}{2-m}\frac{1}{N^{2}}(2N\lambda)^{1/2-n/2}O_2(1) = O_2(\frac{1}{N^{3/2}}),
\end{align*}
as wished.
\end{proof}

%%%%%%%%%%%%%%%%%%%%%%%%%%%%%%%%%%%%%%%%%%%%%%%%
We now proceed to simplify expression (\ref{Derivative oAN}). From this point onward, set $\bar{s}=\sqrt{2N\lambda_0}$, where $\lambda_0<\lambda$.
\begin{Proposition}\label{controlled extra terms}
The following equality holds,
\be\label{Controlled terms}
3[A_{N}] + [B_{N}] + [C_{N}] + [D_{N}] = - [A_N] +O_0(\frac{1}{N^{3/2}}).
\ee
\begin{proof}
    We first show that $[B_N]$ and $[C_N]$ decay exponentially fast. In order to control $[B_N]$, observe that $\hat{\nabla}_{\hat{\rm n}} (|\hat{\nabla} b^{2}|^{2})=\hat{g}(\hat{\nabla}|\hat{\nabla} b^{2}|^{2}, \hat{n})$. Now, we write $|\hat{\nabla} b^{2}|^{2}=4b^2|\hat{\nabla}b|^2$ and use (\ref{expression for grad b2}) to obtain,
        \be\nonumber
        \hat{\nabla}|\hat{\nabla} b^{2}|^{2}= 8b|\hat{\nabla}b|^2\hat{\nabla}b + 4b^2\hat{\nabla} \left(O_2(\frac{1}{N})\right).
        \ee
        Using that $\partial_r = r/N \partial_\tau$ and $b=O_2(N^{1/2})$, we get,
        \be\nonumber
        \hat{\nabla}|\hat{\nabla} b^{2}|^{2} = O_1(N^{1/2})\partial_r + \sum_{i=1}^nO_1(N^{1/2})\partial_{x^i}.
        \ee
        Since by (\ref{GRADORD}) and (\ref{grad b}) we have $\hat{n}= O_2(1)\partial_r + O_2(N^{-1})\nabla f$, we see,
        \be\nonumber
        \lvert \hat{\nabla}_{\hat{\rm n}} (|\hat{\nabla} b^{2}|^{2}) \rvert = O_1(N^{1/2}).
        \ee
        We integrate using (\ref{Vol Elm in s}) to find,
        \be\nonumber
        [B_N] = \frac{c_N}{4s^{m+1}}\int_{b=\bar{s}} O_1(N^{1/2}) d\hat{A} = O_1(N)\left( \frac{\bar{s}}{s}\right)^N = O_1(N) \left( \frac{\lambda_0}{\lambda}\right)^{N/2},
        \ee
        and since $\lambda_0 < \lambda$, it decays exponentially fast.
       
       In order to control $[C_N]$, we observe that since $(\hat{\Delta} b^{2})\hat{\nabla}_{\hat{\rm n}} b^{2}=(\hat{\Delta} b^{2}) 2b|\hat{\nabla}b|$, we can combine this with (\ref{Lap b2}) and (\ref{GRADORD}) to show that, 
       \be\nonumber
       (\hat{\Delta} b^{2})\hat{\nabla}_{\hat{\rm n}} b^{2} = O_2(N^{3/2}),
       \ee
       and now a similar computation as the one performed for $[B_N]$ shows that $[C_N]$ also decays exponentially fast.

        For $[D_N]$ we proceed as follows. By the coarea formula,
        \be\label{expression Dn}
        \begin{split}
        \frac{c_{N}}{2s^{m+1}}\int_{\bar{s}\leq b\leq s} \frac{8m}{2-m}|&\hat{\nabla} b|^{2}b^{m} \hat{\Delta} h\,  d\hat{V}  \\
        &= \frac{1}{s^{m+1}}\frac{4m}{2-m}\int_{\bar{s}}^s w^m \left(c_N\int_{b=w} |\hat{\nabla} b| \hat{\Delta} h \, d\hat{A}\right)dw.
        \end{split}
        \ee
        Then, perform a Taylor expansion on the term multiplying $w^m$ around $s$  to obtain,
        \be\label{taylor Dn}
        \begin{split}
        c_N&\int_{b=w} |\hat{\nabla} b| \hat{\Delta} h \, d\hat{A} \\ 
        &= c_N\int_{b=s} |\hat{\nabla} b| \hat{\Delta} h \, d\hat{A} +(2-m)(w-s)\frac{d}{ds}\left(\frac{c_N}{(2-m)}\int_{b=\xi_{w,N}} |\hat{\nabla} b| \hat{\Delta} h \, d\hat{A}\right),
        \end{split}
        \ee
        for some $\xi_{w,N} \in [w,s]$ (in particular, $\xi_{w,N} \in [\bar{s},s]$), and observe that the term inside parentheses is $[A_N]$ evaluated at $s=\xi_{w,N}.$ Using (\ref{GRADORD}) and (\ref{DhO}), we see that $|\hat{\nabla} b| \hat{\Delta} h = r^{2-m}O_2(1/N)$.
        Then, by (\ref{expansion r}) and (\ref{Vol Elm in s}),
        \be\label{order of A_N}
        \begin{split}
        [A_N]=\frac{c_N}{2-m}\int_{b=s} |\hat{\nabla} b| \hat{\Delta} h d\hat{A} & = \frac{(2N)^{n/2}}{2(4\pi)^{n/2}} \int_M s^{2-m+N}e^{-f}O_2(\frac{1}{N})(1+O_2(\frac{1}{N}))d\nu \\
        & = O_2(\frac{1}{N^{1/2}}).
        \end{split}
        \ee
         We use this to compute,
        \be\label{derivative of AN}
        \frac{d}{ds}[A_N]=\frac{d}{ds}\left(\frac{c_N}{2-m}\int_{b=s} |\hat{\nabla} b| \hat{\Delta} h d\hat{A}\right) = \frac{d\lambda}{ds} \frac{d}{d\lambda} \left( O_2(\frac{1}{N^{1/2}})\right) = O_1(\frac{1}{N}),
        \ee
        and  combine it with (\ref{taylor Dn}) to find,
        \be\nonumber
        c_N\int_{b=w} |\hat{\nabla} b| \hat{\Delta} h \, d\hat{A} = c_N\int_{b=s} |\hat{\nabla} b| \hat{\Delta} h \, d\hat{A} +(w-s)O_1(1).
        \ee
        Using this expression together with (\ref{expression Dn}), we can proceed as in (\ref{Argument for o0 taylor}) and (\ref{Taylor for V_N}) to show,
        \be\label{term AN in DN}
         \frac{c_{N}}{2s^{m+1}}\int_{\bar{s}\leq b\leq s} \frac{8m}{2-m}|\hat{\nabla} b|^{2}b^{m}\hat{\Delta} h\,  d\hat{V} = 4[A_N] + O_0(\frac{1}{N^{3/2}}).
        \ee
        In order to control the remaining term, we write
        \be\nonumber
        \begin{split}
        \frac{c_{N}}{2s^{m+1}}\int_{\bar{s}\leq b\leq s}&\frac{4 b^{2m}}{(2-m)^{2}}(\hat{\Delta}h)^{2}\, d\hat{V} \\ &= \frac{2}{s^{m+1}}\int_{\bar{s}}^s w^m \left( \frac{c_N}{(2-m)^2} \int_{b=w} w^m \frac{(\hat{\Delta}h)^{2}}{|\hat{\nabla} b|} \, d\hat{A}\right)dw,
        \end{split}
        \ee
         and using (\ref{DhO}), (\ref{GRADORD}), (\ref{expansion r}) and (\ref{Vol Elm in s}), we notice that the term in parentheses is of order $O_2(N^{-3/2}).$ Therefore, we proceed as in (\ref{Argument for o0 taylor}) and (\ref{Taylor for V_N}) to show,
        \be\label{final term Dn}
         \begin{split}
        \frac{c_{N}}{2s^{m+1}}\int_{\bar{s}\leq b\leq s}\frac{4 b^{2m}}{(2-m)^{2}}(\hat{\Delta}h)^{2}\, d\hat{V} &= \frac{1}{m+1}O_0(\frac{1}{N^{3/2}}) \\
        & = O_0(\frac{1}{N^{5/2}}).
         \end{split}
         \ee
         Combining (\ref{term AN in DN}) and (\ref{final term Dn}), it follows that
         \be\label{Control DN}
        [D_N] = -4[A_N] + O_0(\frac{1}{N^{3/2}}).
         \ee
         Given that $[B_N]$ and $[C_N]$ decay exponentially fast, substituting (\ref{Control DN}) into the l.h.s. of (\ref{Controlled terms}) completes the proof.
    \end{proof}
\end{Proposition}

%%%%%%%%%%%%%%%%%%%%%%%%%%%%%%%%%%%%%%%%%%%%%%%% 

\begin{Proposition}\label{difference  of normalized av}
We have,
\be\label{difference AN and VN}
 \frac{2(m-1)}{s}(\overline{\mathcal{A}}_{N}-m\overline{\mathcal{V}}_{N})  = 2\frac{d}{ds}\overline{\mathcal{A}}_{N}(s) + O_0(\frac{1}{N^{3/2}}).
    \ee
\begin{proof}
    We use the coarea formula to  write,
    \be\nonumber
   \overline{\mathcal{V}}_{N} = \frac{c_{N}}{s^{m}}\int_{\bar{s}\leq b\leq s}|\hat{\nabla} b|^{4}\, d\hat{V}  = \frac{1}{s^{m}}\int_{\bar{s}}^{s} w^{m-1}\overline{\mathcal{A}}_{N}(w) dw,
    \ee
    and perform a Taylor expansion of $\overline{\mathcal{A}}_{N}$ around $s$, to obtain, 
    \be\label{taylor OAn}
    \overline{\mathcal{A}}_{N}(w) = \overline{\mathcal{A}}_{N}(s) + (w-s)\frac{d}{ds}\overline{\mathcal{A}}_{N}(s) +\frac{(w-s)^2}{2}\frac{d^2\overline{\mathcal{A}}_{N}}{ds^2}(\xi_{w,N}), 
    \ee
    for some $\xi_{w,N} \in [w,s].$ In order to control the second derivative, we notice that
    \be\label{Oan sum}
\overline{\mathcal{A}}_N = \mathcal{A}_N + \frac{c_N}{s^{m-1}}\int_{b=s}| \hat{\nabla} b| d\hat{A}.
    \ee
    By (\ref{dA_N/ds}) and (\ref{second derivative A_N}), we have $d\mathcal{A}_N/ds = O_1(N^{-1/2})$ and $d^2\mathcal{A}_N/ds^2 = O_0(N^{-1})$, respectively.
    Using Proposition \ref{Der normalizing term}, we compute the second derivative of the second term on the r.h.s of (\ref{Oan sum}) as, 
    \be\label{ds2Oan}
    \begin{split}
    \frac{d^2}{ds^2}\left( \frac{c_N}{s^{m-1}}\int_{b=s}| \hat{\nabla} b| d\hat{A}\right) & = \frac{d}{ds}[A_N]+ \frac{d\lambda}{ds}\frac{d}{d\lambda}\left( O_2(\frac{1}{N^{3/2}})\right)\\
    & = O_1(\frac{1}{N}),
    \end{split}
    \ee
    since $d[A_N]/ds$ was already computed in (\ref{derivative of AN}).
    This shows,
    \be\nonumber
    \frac{d^2}{ds^2}\overline{\mathcal{A}}_N = O_0(\frac{1}{N}),
    \ee
    and therefore,
    \be\nonumber
    \overline{\mathcal{A}}_{N}(w) = \overline{\mathcal{A}}_{N}(s) + (w-s)\frac{d}{ds}\overline{\mathcal{A}}_{N}(s) +\frac{(w-s)^2}{2}O_0(\frac{1}{N}).
    \ee
    Now integrate using the bound $|O_0(1/N)| \leq C/N$ for some $C>0$ and proceed as in (\ref{Argument for o0 taylor}) and (\ref{Taylor for V_N}) to show,
    \be\nonumber
    \frac{1}{s^{m}}\int_{\bar{s}}^{s} w^{m-1} \overline{\mathcal{A}}_{N}(w) dw = \frac{1}{m} \overline{\mathcal{A}}_{N}(s) -\frac{s}{m(m+1)}\frac{d}{ds}\overline{\mathcal{A}}_{N}(s) + O_0(\frac{1}{N^{3}}),
    \ee
    where we have absorbed the term decaying exponentially fast into $O_0(1/N^3).$ We finish the proof by expanding the l.h.s. of (\ref{difference AN and VN}) to show,
    \begin{align*}
 \frac{2(m-1)}{s}(\overline{\mathcal{A}}_{N}-m\overline{\mathcal{V}}_{N})  &=  \frac{2(m-1)}{s}\left(\frac{s}{(m+1)}\frac{d}{ds}\overline{\mathcal{A}}_{N}(s) + O_0(\frac{1}{N^2})\right)
\\ &= 2\frac{d}{ds}\overline{\mathcal{A}}_{N}(s) + O_0(\frac{1}{N^{3/2}}).
    \end{align*}
\end{proof}
\end{Proposition}

%%%%%%%%%%%%%%%%%%%%%%%%%%%%%%%%%%%%%%%%%%%%%%%%

Combining the previous results, we obtain the following theorem.

\begin{Theorem}\label{derivative renormalized volume}
We have,
\be\label{derivative of A_N2}
\frac{d}{ds} \mathcal{A}_N = - \frac{c_{N}}{2s^{m+1}}\int_{\bar{s}\leq b\leq s}\left(\left|\hat{\nabla}\hat{\nabla} b^{2}-\frac{\hat{\Delta} b^{2}}{m}\hat{g}\right|^{2}+\hat{\Ric}(\hat{\nabla}b^{2},\hat{\nabla} b^{2})\right)d\hat{V} + O_0(\frac{1}{N^{3/2}}).
\ee
In particular,
    \be\label{derivative monot colding}
    \begin{split}
    \frac{d}{ds}  \big(&2(m-1) \mathcal{V}_N - \mathcal{A}_N\big)(s) \\
    &= - \frac{c_{N}}{2s^{m+1}}\int_{\bar{s}\leq b\leq s}\left(\left|\hat{\nabla}\hat{\nabla} b^{2}-\frac{\hat{\Delta} b^{2}}{m}\hat{g}\right|^{2}+\hat{\Ric}(\hat{\nabla}b^{2},\hat{\nabla} b^{2})\right)d\hat{V} + O_0(\frac{1}{N^{3/2}}). 
    \end{split}
    \ee
    \begin{proof}
        To prove (\ref{derivative of A_N2}), subtract (\ref{Derivative of int nabla b}) from (\ref{Derivative oAN}), and simplify the expression using (\ref{Controlled terms}) and (\ref{difference AN and VN}). To prove (\ref{derivative monot colding}), apply (\ref{derivative non normalized monotonicity Colding}).
    \end{proof}
\end{Theorem}

%%%%%%%%%%%%%%%%%%%%%%%%%%%%%%%%%%%%%%%%%%%%%%%% 

The integrand in (\ref{derivative of A_N2}) and (\ref{derivative monot colding}) can be written as follows.

\begin{Proposition}
The following equality holds,
\begin{equation}\label{integrand dAn}
    \left|\hat{\nabla}\hat{\nabla} b^{2}-\frac{\hat{\Delta} b^{2}}{m}\hat{g}\right|^{2}+\hat{\Ric}(\hat{\nabla}b^{2},\hat{\nabla} b^{2}) = \frac{4b^4}{(2-m)^2} \left|\nabla\nabla f + \Ric -\frac{1}{2\tau}g\right|^2 +O_1(\frac{1}{N}).
\end{equation}
    \begin{proof}
    Computing the partial derivatives of $b^2$, we see
    \begin{align*}
         \partial_{i}b^2 = -\frac{2b^2}{2-m} \partial_i f, \quad \partial_r b^2 = \frac{2b^2}{r} + O_2(\frac{1}{N^{1/2}}), \quad 
         \partial_i \partial_j b^2 = -\frac{2b^2}{2-m} \partial_i \partial_j  f + O_2(\frac{1}{N}).
    \end{align*}
Using  (\ref{Lap b2}) and the fact that $b^2/(2-m) = -2\tau +O_2(1/N)$ we write, 
\begin{equation}\nonumber
    \frac{\hat{\Delta}b^2}{m}\hat{g} = -\frac{2b^2}{2-m}\left(\frac{1}{2\tau}+O_2(\frac{1}{N})\right)\hat{g}.
\end{equation}
To calculate the components of $\hat{\nabla}\hat{\nabla}b^2$, we can refer to the Christoffel symbols listed in Section 3.1 of \cite{CZ06}.
A meticulous yet straightforward inspection  shows that the only  contributions of order $O(1)$ to the squared norm of $$S:=\hat{\nabla}\hat{\nabla}b^2-\frac{\hat{\Delta} b^{2}}{m}\hat{g},$$ arise from the $\{ij\}$ indices during the norm computation. The relevant Christoffel symbols are,
\begin{align}
    \hat{\Gamma}^r_{ij} = -\frac{r}{N}g^{rr}R_{ij}, \quad  \hat{\Gamma}^k_{ij} = \Gamma_{ik}^k, \quad
    \hat{\Gamma}^\alpha_{ij} =0.
\end{align}
We can then show,
\begin{equation}\nonumber
    -\hat{\Gamma}^r_{ij}\partial_r b^2  = \frac{2b^2}{N}R_{ij}+O_2(\frac{1}{N}) = -\frac{2b^2}{2-m}R_{ij} +O_2(\frac{1}{N}),
\end{equation}
and therefore, 
\begin{align*}
    S_{ij} & = \partial_i\partial_j b^2 - \hat{\Gamma}^k_{ij}\partial_k b^2  - \hat{\Gamma}^r_{ij}\partial_r b^2 -\frac{\hat{\Delta} b^{2}}{m}\hat{g}_{ij}\\  
    &= -\frac{2b^2}{2-m}\left(\nabla_i\nabla_j f + R_{ij} - \frac{1}{2\tau}g_{ij}\right) + O_2(\frac{1}{N}).
\end{align*}
Given that $v = 1 -2\tau R/N + O_2(1/N^2),$ it is straightforward to verify that, after crucial cancellations in $S_{\alpha \beta}$ and $S_{rr}$ between $\hat{\nabla}\hat{\nabla}b^2$ and $\hat{\Delta}b^2\hat{g}/m$, the remaining components of $S$ are,
\begin{align*}
S_{\alpha \beta} = O_2(1), \quad S_{rr} = O_2(\frac{1}{N}), \quad S_{ir} = O_2(\frac{1}{N^{1/2}}), \quad S_{i\alpha}=S_{r \alpha} = 0,
\end{align*}
from where we compute,
\begin{align}
\left|\hat{\nabla}\hat{\nabla} b^{2}-\frac{\hat{\Delta} b^{2}}{m}\hat{g}\right|^{2} & =
\hat{g}^{ik}\hat{g}^{jl}S_{ij}S_{kl} + O_2(\frac{1}{N})\\
& = \frac{4b^4}{(2-m)^2} \left|\nabla\nabla f + \Ric -\frac{1}{2\tau}g\right|^2 +O_2(\frac{1}{N}).
\end{align}
Now, using the coordinate expression for $\hat{\Ric}$ in terms of the Christoffel symbols, a straightforward computation (as performed in Corollary 3.1.2 of \cite{CZ06}) reveals that its components with respect to the $(\tau, \theta, x)$-coordinates are of order $O_1(1/N).$  Using (\ref{grad b}) and $\partial_r = (r/N)\partial_\tau$, we show $2b\hat{\nabla}b = O_2(1)\partial_{\tau} + O_2(1)\nabla f$. From this, it follows that
\begin{equation}
    \hat{\Ric}(\hat{\nabla}b^2,\hat{\nabla}b^2) = O_1(\frac{1}{N}).
\end{equation}
    \end{proof}

\end{Proposition}

%%%%%%%%%%%%%%%%%%%%%%%%%%%%%%%%%%%%%%%%%%%%%%%%
Finally, we  reproduce Perelman's formula for the derivative of the $\mathcal{W}$-functional as the limit of the derivatives of $\mathcal{W}_N$. 

\begin{Theorem}\label{Final Theo}
The derivative of Perelman's $\mathcal{W}-$functional is,
\be
\frac{d\mathcal{W}}{d\lambda}(\lambda) = -\int_{M} 2\lambda\left( \left|\nabla\nabla f + \Ric -\frac{1}{2\tau}g\right|^2 \right)(4\pi\lambda)^{-n/2}e^{-f} d\nu.
\ee
    \begin{proof}
Define
\be
F_N := \frac{c_N}{s^{N+4}}\int_{b=s} \left(\left|\hat{\nabla}\hat{\nabla} b^{2}-\frac{\hat{\Delta} b^{2}}{m}\hat{g}\right|^{2}+\hat{\Ric}(\hat{\nabla}b^{2},\hat{\nabla} b^{2})\right)\frac{d\hat{A}}{|\hat{\nabla}b|}.
\ee 
Then, using (\ref{Vol Elm in s}) and (\ref{integrand dAn}), it is straightforward to show that
\be\label{expansion FN}
F_N = 4(4\pi)^{-n/2}(2N)^{n/2-1}\left( \int_{M}  \left|\nabla\nabla f + \Ric -\frac{1}{2\tau}g\right|^2e^{-f}d\nu  +O_1(\frac{1}{N})\right).
\ee
It follows that $F_N = (2N)^{n/2-1}O_1(1)$ and therefore,
\be
\frac{d}{ds}F_N = \frac{d\lambda}{ds} \frac{d}{d\lambda}\left( (2N)^{n/2-1}O_1(1)\right) = (2N)^{n/2-3/2}O_0(1).
\ee
Then, a Taylor expansion of order one centered at $s=\sqrt{2N\lambda}$ shows,
\begin{align*}
    \frac{d}{ds} & \big(2(m-1) \mathcal{V}_N - \mathcal{A}_N\big)\\ & = - \frac{c_{N}}{2s^{m+1}}\int_{\bar{s}\leq b\leq s}\left(\left|\hat{\nabla}\hat{\nabla} b^{2}-\frac{\hat{\Delta} b^{2}}{m}\hat{g}\right|^{2}+\hat{\Ric}(\hat{\nabla}b^{2},\hat{\nabla} b^{2})\right)d\hat{V} + O_0(\frac{1}{N^{3/2}}) \\
    & = -\frac{1}{2s^{m+1}}\int_{\bar{s}}^{s} w^{N+4}F_N(w)dw + O_0(\frac{1}{N^{3/2}})\\
    & = -\frac{1}{2s^{m+1}}\int_{\bar{s}}^{s} w^{N+4}\left(F_N(s)+(w-s)\frac{d}{ds}F_N(\xi_{w,s})\right)dw + O_0(\frac{1}{N^{3/2}}) \\
    & = -\frac{1}{2s^{m+1}}\int_{\bar{s}}^{s} w^{N+4}\left(F_N(s)+(w-s)(2N)^{n/2-3/2}O_0(1)\right)dw + O_0(\frac{1}{N^{3/2}}).\\
\end{align*}
Integrating, bounding the integral of the $O_0(1)$ term by above and below as in (\ref{Argument for o0 taylor}), and absorbing the terms that decay exponentially fast into $O_0(N^{-3/2})$, we find, 
\begin{align*}
    \frac{d}{ds} \big(2(m-1) & \mathcal{V}_N - \mathcal{A}_N\big) \\&= -\frac{s^{3-n}}{2(N+5)}F_N + \frac{s^{4-n}}{(N+6)(N+5)}(2N)^{n/2-3/2}O_0(1) +O_0(\frac{1}{N^{3/2}})\\
    & = -\frac{s^{3-n}}{2(N+5)}F_N + O_0(\frac{1}{N^{3/2}}).
\end{align*}
Finally, since 
\be\nonumber
\frac{d\mathcal{W}_N}{d\lambda}=\frac{d}{d\lambda}\big(2(m-1) \mathcal{V}_N - \mathcal{A}_N\big) = \frac{ds}{d\lambda}\frac{d}{ds}  \big(2(m-1) \mathcal{V}_N - \mathcal{A}_N\big),
\ee\nonumber
we use expression (\ref{expansion FN}) and $s=\sqrt{2N\lambda}$ to obtain,
\begin{align*}
\frac{d\mathcal{W}_N}{d\lambda}
& =  \left(\frac{N}{2\lambda}\right)^{1/2}\left(-\frac{(\sqrt{2N\lambda})^{3-n}}{2(N+5)}F_N + O_0(\frac{1}{N^{3/2}})\right)+ O_0(\frac{1}{N^{3/2}})\\
& = -2(4\pi)^{-n/2}(\sqrt{2N\lambda})^{2-n}(2N)^{n/2-1}\left( \int_{M}  \left|\nabla\nabla f + \Ric -\frac{1}{2\tau}g\right|^2e^{-f}d\nu  +O_1(\frac{1}{N})\right) \\
& \phantom{= }+O_0(\frac{1}{N}) \\
& = -2(4\pi)^{-n/2}\lambda^{1-n/2}\left( \int_{M}  \left|\nabla\nabla f + \Ric -\frac{1}{2\tau}g\right|^2e^{-f}d\nu  +O_1(\frac{1}{N})\right)+O_0(\frac{1}{N}) \\
& = -\int_{M} 2\lambda \left( \left|\nabla\nabla f + \Ric -\frac{1}{2\lambda}g\right|^2 \right)(4\pi\lambda)^{-n/2}e^{-f} d\nu +O_0(\frac{1}{N}).
\end{align*}
By Corollary \ref{convergence monot Colding},
\be\nonumber
\frac{d\mathcal{W}}{d\lambda}(\lambda) = \lim_{N \to \infty} \frac{d\mathcal{W}_N}{d\lambda}(\lambda),
\ee
which finishes the proof.
\end{proof}

\end{Theorem}

\bibliographystyle{plain}
\bibliography{bibliography}
\end{document}